\documentclass [10pt]{amsart}

\usepackage{amsmath}
\usepackage{amsfonts}
\usepackage{amstext}
\usepackage{amsbsy}
\usepackage{amsopn}
\usepackage{amsxtra}
\usepackage{upref}
\usepackage{amsthm}
\usepackage{amsmath}
\usepackage{amssymb}
\usepackage{tikz}
\usetikzlibrary{automata,positioning}
\usepackage{enumitem}
\usepackage{epsfig,verbatim,ifthen,graphicx}
\usepackage[all]{xy}
\usepackage{epsfig,verbatim,ifthen,graphicx}

\newtheorem{prop}{Proposition}[section]
\newtheorem{lemma}{Lemma}[section]
\newtheorem{thm}{Theorem}[section]
\newtheorem{ex}{Example}[section]
\newtheorem{coro}{Corollary}[section]

\theoremstyle{definition}
\newtheorem{defi}{Definition}[section]
\newtheorem{rem}{Remark}[section]
\usepackage{enumitem}

\def\R{{\mathbb R}}

\def\N{{\mathbb N}}

\def\F{{\mathcal F}}

\def\G{{\mathcal G}}
\def\H{{\mathcal H}}

\begin{document}
\title
{On Gibbs measures for almost additive sequences associated to some relative pressure functions} 


\author{Yuki Yayama}
\address{Centro de Ciencias Exactas, 
Departamento de Ciencias B\'{a}sicas,  Universidad del B\'{i}o-B\'{i}o, Avenida Andr\'{e}s Bello 720, Casilla 447, Chill\'{a}n, Chile}

\begin{abstract}
Given a 
weakly almost additive sequence of continuous functions with bounded variation $\F=\{\log f_n\}_{n=1}^{\infty}$ on a subshift $X$ over finitely many symbols, we study properties of a function $f$ on $X$ such that $\lim_{n\to\infty}\frac{1}{n}\int \log f_n d\mu$
\\$=\int f d\mu$  for every invariant measure $\mu$ on $X$. Under some conditions we construct a function $f$ on $X$ explicitly and study a relation between the property of $\F$ and some particular types of $f$.  As applications we study images of Gibbs measures for continuous functions under one-block factor maps. 
We  investigate a relation between the almost additivity of the sequences associated to relative pressure functions and the fiber-wise sub-positive mixing property of a factor map.
For a special type of one-block factor maps between shifts of finite type, we study  necessary and sufficient conditions for the image of a one-step Markov measure to be a Gibbs measure for a continuous function.\\




\end{abstract}
\maketitle
\section{Introduction}\label{intro}
The thermodynamic formalism for sequences of continuous functions has been developed
in connection to the study of dimension problems. 
In particular, 
Barreira \cite{b2} and Mummert \cite {m} introduced almost additive sequences which generalize continuous functions and developed the formalism for such sequences. More generally, asymptotically additive sequences which generalize almost additive sequences were introduced by Feng and Huang \cite{FH}. Cuneo \cite{Cu} showed that if $\F=\{\log f_n\}_{n=1}^{\infty}$ is asymptotically additive (see Section \ref{back} for the definition), then an equilibrium state for $\F$ is an equilibrium state for a continuous function.

Let $(X, \sigma_X)$ be a subshift over finitely many symbols.
If an almost additive sequence of continuous functions $\F$ on $X$ with the weak specification property has bounded variation, then it has a unique equilibrium state which is also a unique invariant Gibbs measure for $\F$  \cite{b2, m}.  Since an almost additive sequence $\F=\{\log f_n\}_{n=1}^{\infty}$  on $X$ is asymptotically additive, there exists $\hat f\in C(X)$ such that  
\begin{equation}\label{cuneothm1}
\lim_{n\rightarrow \infty} \frac{1}{n} 
\lVert \log f_{n}-S_n \hat f \rVert_{\infty}=0,
\end{equation}
where $(S_n \hat f)(x):=\sum_{i=0}^{n-1}\hat f(\sigma_X^{i}(x))$ and $\vert\vert \cdot\vert \vert_{\infty}$ is the supremum norm \cite{Cu}. Let $M(X, \sigma_X)$ the set of all $\sigma_X$-invariant Borel probability measures on $X$. Then 
\begin{equation} \label{propas}
\lim_{n\to \infty} \frac{1}{n}\int \log f_n dm=\int \hat f dm \text{ for every } m\in M(X, \sigma_X) \end{equation}
and $\F$ and $\hat f$ have the same equilibrium states.
If an asymptotically additive sequence $\F$ has an equilibrium state $\mu$ which is Gibbs for $\F$ and $\mu$ is also Gibbs for $\hat f\in C(X)$ in (\ref {cuneothm1}) 
then the sequence $\F$ is almost additive. 
The question we pose is as follows. 

[Q1] What form does $\hat f$ take when $\F$ is (weakly) almost additive?  What are the properties of $\hat f$? We note that such an $\hat f$ is not unique. 

For this purpose, we first investigate a natural form of a Borel measurable function $\hat f$ on $X$ satisfying (\ref{cuneothm1}). We extend the notion of the (weak) Gibbs measure for a continuous function to a Borel measurable function (see Section \ref{back}). 

In Theorem \ref{go1}, under certain conditions we construct a Borel measurable function  $\hat f$  satisfying (\ref{cuneothm1}) explicitly and characterize the Gibbs and weak Gibbs properties of equilibrium states for such functions $\hat f$ by some properties of the sequences $\F$. 
For an almost additive sequence $\F$ with bounded variation with a unique equilibrium state $\mu$, under certain conditions, 
we find a function $\hat f$ satisfying (\ref{cuneothm1}) for which  $\mu$ is  a unique invariant Gibbs measure. 
If $\F$ is merely weakly almost additive and has an equilibrium state $\mu$ which is Gibbs for $\F$ (such a case does exist, see Remark \ref{obs}), then we find a function $\hat f$ satisfying (\ref{cuneothm1}) such that $\mu$ is an equilibrium state  for $\hat f$ and is an invariant weak Gibbs measure for $\hat f$.
In Theorem  \ref{cando1}, by studying a unique invariant Gibbs measure for a (weakly) almost additive additive sequence we obtain a similar result. 

In Section \ref{apli}, 
we study questions on factors of Gibbs measures. Given a one-block factor map  $\pi: X\to  Y$  between subshifts and  a Gibbs measure $\mu\in M(X,\sigma_X)$ for $f\in C(X)$,  we are interested in  studying  some properties of the push-forward measure $\pi\mu$ (which we call an image or a factor of $\mu$). Many studies have been conducted on the question on factors of Gibbs measures (see for example, \cite{CU1, CU2, PK, K, Yo, Pi, Pi2}).  In particular, Yoo \cite{Yo} showed that the fiber-wise sub-positive mixing property of factor maps  (Definition \ref{subp}) 
is a sufficient condition for a factor map to send all Markov measures on a topologically mixing shift of finite type to Gibbs measures and the result was extended by Piraino \cite{Pi2} to the Gibbs measures associated to continuous functions. However, 
the fiber-wise sub-positive mixing property is not a necessary condition 
(see Proposition \ref{aafiber}).

In this paper, we study
necessary and sufficient conditions for  $\pi\mu$  to be Gibbs for a continuous (or more generally Borel measurable) function under a particular setting, by an explicit construction of a function.  We also consider the case when the image is weak Gibbs for a continuous function.
To this end, we use the non-additive thermodynamic formalism, studying  subadditive sequences $\G=\{\log g_n\}_{n=1}^{\infty}$ associated to relative pressure functions (see for example \cite{Fe1, Y2, Y3}) and apply Theorem \ref{go1}. 

 Let $(X, \sigma_X), (Y, \sigma_Y)$ be subshifts and  
$\pi:X\rightarrow Y$ be a factor map.
Let $f\in C(X)$, $n\in \N$ and $\delta>0$. For each $y\in Y$, define
\begin{eqnarray}\label{relatived}
\begin{split}
P_n(\sigma_X, \pi, f, \delta)(y)&=\sup\{\sum_{x\in E}e^{(S_nf)(x)}:E \text { is an } (n, \delta) \text{ separated subset of } \pi^{-1}(\{y\})\},
\\
&P(\sigma_X, \pi, f, \delta)(y)=\limsup_{n\rightarrow \infty}\frac{1}{n}\log P_{n}(\sigma_X, \pi, f,\delta)(y),\\
&P(\sigma_X, \pi, f)(y)=\lim_{\delta\rightarrow 0}P(\sigma_X, \pi,f, \delta)(y).
\end{split}
\end{eqnarray}
The function $P(\sigma_X, \pi, f):Y\rightarrow \R$, introduced by Ledrappier and Walters \cite{LW}, 
is the {\em relative pressure} function of $f\in C(X)$ with respect to $\pi$. 
In general it is merely Borel measurable. 
We have
$P(\sigma_X, \pi, f) =\lim_{n\rightarrow\infty}(1/n)\log g_n$ 
almost everywhere with respect to every $\mu \in M(Y, \sigma_Y)$, 
where $g_n$ is defined in (\ref{beq}). 
The sequence $\G=\{\log g_n\}_{n=1}^{\infty}$ on $Y$  is subadditive in general and is the {\em sequence associated to the relative pressure function $P(\sigma_X, \pi, f)$}.

[Q1] and questions concerning factors of Gibbs measures are related.
Section \ref{apli} and the rest of the paper is devoted to finding an explicit form for $\hat f$ satisfying (\ref{cuneothm1}) when $\F$ is $\G$. 
In Section \ref{fiberaaR} we consider the following type of factor maps to study the equivalent conditions for an images of the Gibbs measure associated to a continuous function to be Gibbs for a continuous function. Let $\pi: X\to Y$ be a one-block factor map between irreducible shifts of finite type where 
$X\subseteq \{1,2,3\}^{\N}$ and $Y\subseteq\{1,2\}^{\N}$ such that $Y$ is the full shift or has the transition matrix $\begin{pmatrix}
0 & 1 \\
1 & 1 
\end{pmatrix}$, satisfying $\pi^{-1}\{1\}=\{1\}$. The images of measures of maximal entropy under the factor maps of this type were originally studied in \cite{S1, Y1}. In this setting, by Theorem \ref{go1} with some computations, we can find an explicit form for the desired function.
For  a function of two coordinates $f\in C(X)$, 
in Theorem \ref{phd1} and Corollary \ref{phd2}, we study necessary and sufficient conditions for the image of a
Gibbs measure $\mu$ associated to $f$ to be a Gibbs measure for a continuous function  by using the almost additivity of the sequence $\G$ associated to 
$P(\sigma_X, \pi, f)$.  Hence the results are applied to Markov measures $\mu$.  
This enables us to answer [Q1] for some cases. 
Given a factor map $\pi$ and $f\in C(X)$ described above, set $\F=\G$ and $f_n=g_n$ in (\ref{cuneothm1}). 
A function  $\hat f$ obtained from our approach is continuous on $Y$ except for a particular case which 
is left for future study (see Corollary \ref{best} and Lemma \ref{exception}).

Section 5 consists of the proof of Theorem \ref{phd1}. We apply some ideas from \cite{CU1, Yo} to connect $\G$ with the sum of all entries of product of certain matrices. 
To study the sum, we use real Jordan canonical forms of submatrices. 
In Section 6, we give some examples of factor maps from  \cite{S1, Y1} which illustrate Theorem \ref{go1} (Corollary \ref{aplrelative}) and Theorem \ref{phd1}. Finally, we are left with a question: Can we generalize the results (Theorems \ref{phd1}) to the images of Gibbs measures for continuous functions under a more general setting? 

We remark that in \cite{CE} given an almost additive $\F$ with bounded variation some properties of a function $\hat f$ satisfying (\ref{cuneothm1}) were studied. 
An explicit form for $\hat f$ was not studied.




\section{Background}\label{back}
\subsection{Shift spaces}
We give a brief summary of the basic definitions in symbolic dynamics. 
$(X, \sigma_X)$ is a {\em one-sided subshift} over $\{1,\dots, k\}$ if $X$ is a closed
shift-invariant subset of $\{1,\dots, k\}^{\N}$ for some $k\geq
1$, i.e., $\sigma_X(X)\subseteq X$,  where the shift 
$\sigma_X:X\rightarrow X$
is defined by $(\sigma_X(x))_{i}=x_{i+1}$ for all $i\in \N$, $x=(x_n)^{\infty}_{n=1} \in X.$
Define a metric $d$ on $X$ by $d(x,x')={1}/{2^{k}}$ if
$x_i=x'_i$ for all $1\leq i\leq k$ and $x_{k+1}\neq {x'}_{k+1}$, $d(x,x')=1$ if $x_1\neq x'_1$,
and $d(x,x')=0$ otherwise. 
Define a cylinder set $[x_1 \dots x_{n}]$ of length $n$ in $X$  by 
$[x_1\dots x_n]=\{(z_i)_{i=1}^{\infty} \in X:  z_i=x_i \text{ for all }1\leq i\leq n\}.$  For each $n \in \N,$ denote by
$B_n(X)$ the set of all $n$-blocks that appear in points in $X$.
Define $B_{0}(X)=\{\epsilon\},$ where  $\epsilon$ is the empty word of length $0$.
The language of $X$ is the set $B(X)=\cup_{n=0}^{\infty}B_n(X)$. 
A subshift $(X,\sigma_X)$ is {\it irreducible} if for any allowable words $u, v\in B(X)$, there exists  $w\in B(X)$ such that $uwv \in B(X)$, and has the {\it weak specification property} if there exists $p\in\N$ such that for any allowable words $u, v \in B(X)$, there exist  $0\leq k\leq p$ and $w\in B_{k}(X)$ such that $uwv\in B(X)$. 
A point $x\in X$ is a periodic point  of $\sigma_X$ if there exists $p\in\N$ such that $\sigma_X^{p}(x)=x$. Let 
$A=(a_{ij})$ be a $k\times k$ matrix of zeros and ones. Define $X_A$ by 
$$X_A= \left\{ (x_n)_{n =1}^{\infty} \in \{1,\dots, k\}^{\N}
: a_{x_{n}, x_{n+1}}=1  \text{ for every } n \in \N \right\}.$$
Then  $(X_A, \sigma_{X_{A}})$ is a {\em one-sided shift of finite type} and it is a subshift over $\{1,\dots, k\}$.  
It is topologically mixing if there exists $p\in \N$ such that $A^p>0$.

Let  $(X, \sigma_X)$ be a subshift over a finite set  $S_1$ and $(Y, \sigma_Y)$ be a subshift over a finite set  $S_2$. 
A map $\pi:X\rightarrow Y$ is a {\it factor map} if it is continuous, surjective and satisfies $\pi \circ \sigma_{X} = \sigma_Y\circ \pi$.  
A one-block code is a map $\pi : X\to Y$ for which there exists a function $\tilde \pi:S_1(X) \rightarrow S_2(Y)$  such that $(\pi (x))_i = \tilde \pi (x_i)$,  for all $i \in \N$.
 
 \subsection{Sequences of continuous functions.}\label{seqmany}
  Given a subshift $(X, \sigma_X)$,  for each $n\in\N$, let $f_n: X\rightarrow \R^{+}$ be a continuous function. Then $\F=\{\log f_n\}_{n=1}^{\infty}$ is a sequence of continuous functions on $X$.
A sequence $\F$ is {\em almost additive} on $X$ if there exists a constant $C> 0$ 
such that 
\begin{equation}\label{aa1}
 f_{n+m}(x) \leq e^{C}f_n(x) f_{m}(\sigma^n_X x)
\end{equation}
and 
\begin{equation}\label{aa2}
f_{n+m}(x) \geq e^{-C}f_n(x) f_{m}(\sigma^n_X x) 
\end{equation}
for every $x\in X$, $n,m\in\N$. 
More generally, a sequence $\F$ is {\em weakly  almost additive} if there exists a sequence of positive real numbers $\{C_n\}_{n=1}^{\infty}$ satisfying $\lim_{n\to \infty}(1/n)C_n=0$ such that (\ref{aa1}) and (\ref{aa2}) hold if we replace $C$
 by $C_n$ \cite {CJ}.
A sequence $\mathcal{F}$ is {\it subadditive} if $\F$ satisfies (\ref{aa1}) with $C=0$ \cite{CFH}. 
Asymptotically additive sequences were introduced in \cite{FH} and we use the following definition (see \cite{Cu}).  A sequence $\F$ is {\em asymptotically additive} on $X$ if there exists $\hat f\in C(X)$ such that (\ref{cuneothm1}) holds.
 Weakly almost additive sequences are asymptotically additive \cite[Lemma 6.2]{CJ}. 
A sequence $\mathcal{F}= \{ \log f_n \}_{n=1}^{\infty}$ has {\it bounded variation}  if there exists $M \in \R^{+}$ such that
 $\sup \{ M_n : n \in \N\} \leq M$ where
 $M_n= \sup \{ {f_n(x)}/{f_n(y)} : x,y  \in X, x_i=y_i \textrm{ for }$\\
 $1 \leq i \leq n\}.$ Given a sequence of continuous functions $\F=\{\log f_n\}_{n=1}^{\infty}$ on $X$, 
a measure $\mu\in M(X, \sigma_X)$ is an {\em equilibrium state} for $\F$ if 
\begin{equation}
h_{\mu}(\sigma_X)+\lim_{n\rightarrow\infty}\frac{1}{n}\int\log f_nd\mu
=\sup_{m\in M(X,\sigma_X)}\{h_{m}(\sigma_X)+\lim_{n\rightarrow\infty}\frac{1}{n}\int\log f_ndm\}.
\end{equation} 
A measure $\mu\in M(X, \sigma_X)$ is a {\em weak Gibbs measure} for 
$\F$ if there exist $P\in \R$ and $C_n>0$ satisfying $\lim_{n\rightarrow\infty}(1/n)\log C_n=0$ such that 
\begin{equation}\label{gibbsd}
\frac{1}{C_n}\leq \frac{\mu[x_1\dots x_n]}{e^{-nP}f_n(x)}\leq C_n
\end{equation}
for every $x\in X$ and $n\in \N$. If there exists $C>0$ such that $C_n=C$ for all $n\in\N$, then  $\mu$ is a  {\em Gibbs measure} for $\F$. 
If $\F$ is almost additive with bounded variation on a subshift with the weak specification property,  then $\F$ has a unique equilibrium state and it is the unique invariant Gibbs measure for $\F$ (see \cite{b2, m, Fe1}).  We note that if $\mu$ is a weak Gibbs measure for an asymptotically additive sequence $\F$, then the variational principle of topological pressure holds and $P$ in (\ref{gibbsd}) is given by the topological pressure of $\F$ (see for example \cite{FH}).
Similarly, given a Borel measurable function $g$ on $X$, a measure $\mu\in M(X, \sigma_X)$ is an {\em equilibrium state} for $g$ if 
$h_{\mu}(X)+\int g d\mu= \sup_{m\in M(X,\sigma_X)}\{h_{m}(\sigma_X)+\int gdm\}.$ 
A measure $\mu\in M(X, \sigma_X)$ is a {\em weak Gibbs measure} for $g$ if there exist $P\in \R$ and $C_n>0$ such that 
(\ref{gibbsd}) holds if we replace $f_n(x)$ by $e^{(S_ng)(x)}$  and it is a {\em Gibbs measure} for $g$ if $C_n=C$ for all $n\in \N$ for some $C>0.$ 

\section{Gibbs measures for weakly almost additive sequences}\label{resu}
In this section, given an almost additive sequence $\F$ of continuous functions on a subshift with the weak specification property, we construct a Borel measurable function $\hat f$ satisfying (\ref{cuneothm1}) under certain conditions and study the Gibbs property of the unique equilibrium state for $\hat f$ when $\F$ has bounded variation. We also consider the case when  $\F$ is weakly almost additive. 
We first start with a simple lemma. 

\begin{lemma} \label{bounded}
Let $(X,\sigma_X)$ be a subshift with the weak specification property and $\F=\{\log f_n\}_{n=1}^{\infty}$ be a sequence of continuous functions on $X$. 
Suppose that there exist a Borel measurable function $f$ on $X$ and  a sequence of positive real numbers $\{A_k\}_{k=1}^{\infty}$ satisfying $\lim_{k\to \infty}(1/k)\log A_k=0$ such that 
\begin{equation}\label{simpimp}
\frac{1}{A_k}\leq \frac{f_k(x)}{e^{(S_kf)(x)}}\leq A_k \text{ for every } x\in X, k\in \N.
\end{equation}
Then  the following hold.
\begin{enumerate}[label=(\roman*)]
\item \label{wp1} If $\mu$  is a Gibbs measure for $\F$, then it is a weak Gibbs measure for $f$.
\item \label{sp2} If there exists a constant $A>0$ such that $A_k=A$ for every $k\in \N$, then 
 $\mu$  is a Gibbs measure for $\F$ if and only if it  is  a Gibbs measure for $f$. \end{enumerate}
\end{lemma}
\begin{proof}
If $\mu$ is a  Gibbs measure for $\F$, then 
there exist $P\in \R, C_0>0$ such that 
\begin{equation*}
\frac{1}{C_0}\leq \frac{\mu[x_1\dots x_n]}{e^{-nP}f_n(x)}\leq C_0 \textnormal{ for every  }x\in X, n\in \N.
\end{equation*}
Replacing $C_n$ and $f_n$ in (\ref{gibbsd}) by $A_nC_0$ and  $e^{S_nf}$ respectively, $\mu$ is weak Gibbs for $f$. The second statement follows by similar arguments,  replacing $A_k$ by $A$ in (\ref{simpimp}). 
\end{proof}


\noindent \textbf{[Setting (A)] }Let $(X,\sigma_X)$ be a subshift with the weak specification property and $\F=\{\log f_n\}_{n=1}^{\infty}$ be a sequence of continuous functions on $X$. 
Define for $x \in X$ $$\underline{f}(x):=\liminf_{n\to \infty}\log \bigg(\frac{f_n(x)}{f_{n-1}(\sigma_X(x))}\bigg) \text{ and } 
\overline{f}(x):=\limsup_{n\to \infty}\log \bigg(\frac{f_n(x)}{f_{n-1}(\sigma_X(x))}\bigg).$$ 

Then $\underline{f}$ and $\overline{f}$ are Borel measurable functions on $X$. In the next theorem we construct a Borel measurable function $\hat f$ satisfying (\ref{cuneothm1}) under certain conditions and study the properties of equilibrium states of $\hat f$. 

\begin{thm} \label{go1}
Under Setting (A), suppose that $\underline{f}(x)=\overline{f}(x)$ for every $x\in X$ and define  $\hat f(x):= \lim_{n\to \infty}\log \bigg(\frac{f_n(x)}{f_{n-1}(\sigma_X(x))}\bigg)$. 
Then the following hold. 
\begin{enumerate}[label=(\roman*)]
\item \label{v1}
If $\F$ is almost additive on $X$, then $f$ is a bounded  Borel measurable function satisfying (\ref{cuneothm1})
and (\ref{propas}). 
If, in addition, $\F$ has  bounded variation on $X$,  then the unique equilibrium state for $\F$ which is also the unique invariant Gibbs measure for $\F$ is the unique equilibrium state for $\hat f$ and it is the unique invariant Gibbs measure for $\hat f$.\\
\item \label{v2}
If $\F$ is merely weakly almost additive on $X$, then $\hat f$ is a bounded  Borel measurable function satisfying (\ref{cuneothm1}) and (\ref{propas}).
 If an equilibrium state for $\F$ is Gibbs for $\F$,  then it is an equilibrium state for $\hat f$ and it is also an invariant weak Gibbs measure for  $\hat f$.
\end{enumerate}  
\end{thm}
\begin{rem}\label{obs}
A weakly almost additive sequence $\F$ on $X$ has an equilibrium state. See Example  \ref{ex3} for the case when such an $\F$ has a unique equilibrium state with the Gibbs property. \end{rem}

We apply the following lemma to show Theorem \ref{go1}. 

\begin{lemma}\label{prego1}
Under Setting (A), the following hold. 
\begin{enumerate}[label=(\roman*)]
\item \label{co2} If  $\F$ satisfies (\ref{aa1}), 
then $\underline{f}$ is bounded from above and $f_k(x)/e^{(S_k \underline{f})(x)}\geq e^{-C}$ for every 
$x\in X$ and $k\in \N$.\\
\item \label{co1} If $\F$ satisfies (\ref{aa2}), then $\overline{f}$ is bounded from below and $f_k(x)/e^{(S_k\overline{f})(x)}\leq e^C$ for every $ x\in X$ and $k\in \N.$\\
\end{enumerate}
\end{lemma}
\begin{proof}
We first  we show \ref{co2}. Since $\F$ satisfies (\ref{aa1}), $\underline{f}$ is bounded from above by $C+\max_{x\in X}\log f_1(x)$. Fix $k\in\N$. Then 
\begin{equation*}
\begin{split}
(S_k\underline{f})(x)&=\liminf_{n\to \infty} \log \bigg(\frac{f_n(x)}{f_{n-1}(\sigma_X(x))} \bigg)+
\liminf_{n\to \infty} \log \bigg(\frac{f_{n-1}(\sigma_X(x))}{f_{n-2}(\sigma_X^2(x))}\bigg)\\
&+\dots  
+\liminf_{n\to \infty} \log \bigg(\frac{f_{n-k+1}(\sigma_X^{k-1} (x))}{f_{n-k}(\sigma_X^k(x))}\bigg)\\
&
\leq \liminf_{n\to \infty} \log \bigg(\frac{f_n(x)f_{n-1}(\sigma_X(x))\dots f_{n-k+1}(\sigma_X^{k-1}(x))}{
f_{n-1}(\sigma_X (x))f_{n-2}(\sigma_X^2 (x))\dots f_{n-k}(\sigma_X^{k}(x))}\bigg)\\ &
=\liminf_{n\to \infty} \log \bigg(\frac{f_n(x)}{f_{n-k}(\sigma_X^k (x))}\bigg)\leq \log (e^{C}f_k(x)),
\end{split}
\end{equation*}
where the last inequality holds by (\ref{aa1}). 
We next show \ref{co1}. Since $\F$ satisfies (\ref{aa2}), $\overline{f}$ is bounded from below by $-C+\min_{x\in X} \log f_1(x)$. Fix $k\in\N$. Then 
\begin{equation*}
\begin{split}
(S_k\overline{f})(x)&=\limsup_{n\to \infty} \log \bigg( \frac{f_n(x)}{f_{n-1}(\sigma_X(x))}\bigg) +
\limsup_{n\to \infty}\log \bigg( \frac{f_{n-1}(\sigma_X (x))}{f_{n-2}(\sigma_X^2(x))}\bigg)\\
&+\dots  
+\limsup_{n\to \infty} \log  \bigg(\frac{f_{n-k+1}(\sigma_X^{k-1}(x))}{f_{n-k}(\sigma_X^k(x))}\bigg)\\
&
\geq \limsup_{n\to \infty}  \log \bigg( \frac{f_n(x)f_{n-1}(\sigma_X (x))\dots f_{n-k+1}(\sigma_X^{k-1}(x))}{
f_{n-1}(\sigma_X(x))f_{n-2}(\sigma_X^2 (x))\dots f_{n-k}(\sigma_X^{k}(x))}\bigg)\\ &
=\limsup_{n\to \infty} \log \bigg(\frac{f_n(x)}{f_{n-k}(\sigma_X^k (x))}\bigg)\geq \log (e^{-C}f_k(x)),
\end{split}
\end{equation*}
 where the last inequality holds by (\ref{aa2}). This proves \ref{co1}.
\end{proof}

\noindent \textbf{[Proof of Theorem \ref{go1}]}
We first show \ref{v1}. Lemma \ref{prego1} implies that $\hat f$ is a measurable  bounded function satisfying (\ref{cuneothm1}) and hence  (\ref{propas}) holds. 
By  (\ref{propas}), the unique equilibrium state $\mu$ for $\F$ which is also a unique invariant Gibbs measure for $\F$ is the unique equilibrium state for $\hat f$. By Lemma \ref{bounded}\ref{sp2}, $\mu$ is an invariant Gibbs measure for $\hat f$. 
To show that it is a unique invariant Gibbs measure for $\hat f$, suppose that there is another invariant Gibbs measure $\tilde \mu\neq \mu$ for $\hat f$. Then by Lemma \ref{bounded}\ref{sp2}
$\tilde \mu$ is also an invariant Gibbs measure for $\F$, which is a contradiction.
 Next we show  \ref{v2}. Since $\F$ is weakly almost  additive, we replace $C$ of $(\ref{aa1})$ and $(\ref{aa2})$ by $C_n>0$ such that  $\lim_{n\to \infty}(1/n)C_n=0$. 
 Then 
$-C_1+\min_{x\in X}\log f_1(x)\leq \hat f(x)\leq -C_1+\max_{x\in X}\log f_1(x)$ for all $x\in X$.
 The inequalities in Lemma \ref{prego1} \ref{co2} \ref{co1} replacing $C$ by $C_k$ hold. Hence  (\ref{cuneothm1}) and (\ref{propas}) hold. Since an equilibrium state for $\F$ is an equilibrium state for $\hat f$,  Lemma \ref{bounded} \ref{wp1} implies the last statement. \\


Next, assuming the existence of an invariant Gibbs measure for a sequence $\F$, we construct a Borel measurable function $\hat f$ satisfying (\ref{cuneothm1}).\\

\noindent \textbf{[Setting (B)]} Let $(X,\sigma_X)$ be a subshift with the weak specification property and $\F=\{\log f_n\}_{n=1}^{\infty}$ be a sequence of continuous functions on $X$. Suppose that there exists an invariant Gibbs measure $\nu$ for $\F$, that is, there exist $P\in \R, C_0>0$ such that 
\begin{equation}\label{GibbsF}
\frac{1}{C_0}\leq \frac{\nu[x_1\dots x_n]}{e^{-nP}f_n(x)}\leq C_0  \text {  for every } x\in X, n\in \N. 
\end{equation}

Note that (\ref{GibbsF}) implies that $\F$ has bounded variation. Define for $x\in  X$ $$\underline{r}(x):=\liminf_{n\to \infty}\log \bigg(\frac{\nu[x_1\dots x_n]}{\nu[x_2\dots x_n]}\bigg)+P 
\text{ and } 
\overline{r}(x):=\limsup_{n\to \infty}\log \bigg(\frac{\nu[x_1\dots x_n]}{\nu[x_2\dots x_n]}\bigg) +P.$$ 
Then $\underline{r}$ and $\overline{r}$ are Borel measurable functions on $X$. 

\begin{thm}\label{cando1}
Under Setting (B), suppose that $\underline{r}(x)=\overline{r}(x)$ for every $x\in X$ and define  $r(x):=\lim_{n\to \infty}\log \bigg(\frac{\nu[x_1\dots x_n]}{\nu[x_2\dots x_n]}\bigg)+P$. 
\begin{enumerate}[label=(\roman*)]
\item \label{f1}If $\F$ is almost additive on $X$, then $r$ is  bounded and Borel measurable  on $X$ and  
\begin{equation}\label{simi1}
\lim_{n\to \infty}\frac{1}{n}\int\log f_n  d\mu  =\lim_{n\to \infty}\frac{1}{n}\int\log (\nu[x_1\dots x_n]e^{np})  d\mu=\int r d\mu
\end{equation}
 for every  $\mu\in M(X,\sigma_X)$.
Then the unique equilibrium state $\nu$ for $\F$ is the unique equilibrium state for $r$ which is also the unique invariant Gibbs measure for $r$. 
 \item \label{f2}If $\F$ is merely weakly additive, then $r$ is bounded and Borel measurable, and (\ref{simi1}) holds. The measure $\nu$ is an equilibrium state for $r$ which is an invariant weak Gibbs measure for $r$. 
 \end {enumerate}
\end{thm}
\begin{proof} 
We first notice that the first equality of (\ref{simi1}) follows from  (\ref{GibbsF}). 
Now we show \ref{f1}.
For $x\in X, n\in \N$, define $r_n(x):=\nu[x_1\dots x_n]e^{nP}$. If $\F$ is almost additive satisfying (\ref{aa1}) and (\ref{aa2}), then
using (\ref{GibbsF}), a simple computation shows that
 $$(C^3_{0}e^C)^{-1}r_{n}(x)r_{m}(\sigma^nx)\leq r_{n+m}(x)\leq r_{n}(x)r_{m}(\sigma^nx)C^3_{0}e^C$$ for every $x\in X, n, m \in \N$. 
 Hence $\{\log r_n\}_{n=1}^{\infty}$ is almost additive and bounded variation. 
 Since $\underline{r}=\overline{r}$, 
applying Theorem \ref{go1}\ref{v1} to $\{\log r_n\}_{n=1}^{\infty}$, we obtain the second equality of (\ref{simi1}) and the first part of \ref{f1}. Notice that  $\nu$ is the unique equilibrium state for $\F$ and 
$\nu$ is Gibbs for $\{\log r_n\}_{n=1}^{\infty}$ because
$\nu[x_1\dots x_n]/(e^{-nP}r_n(x))=1$. Hence it is the unique equilibrium state for $\{\log r_n\}_{n=1}^{\infty}$. Applying Theorem \ref{go1}\ref{v1}, we obtain \ref{f1}.
We next show \ref{f2} in the similar manner. If $\F$ is weakly almost additive, then (\ref{GibbsF}) implies that there exists $C_n>0$ such that $\lim_{n\to \infty}C_n/n=0$ satisfying
$$(C^3_{0}e^{C_n})^{-1}r_{n}(x)r_{m}(\sigma^nx)\leq r_{n+m}(x)\leq r_{n}(x)r_{m}(\sigma^nx)C^3_{0}e^{C_n}$$
for every $x\in X, n, m\in \N$. Hence $\{\log r_n\}_{n=1}^{\infty}$  is weakly almost additive. Theorem \ref{go1}\ref{v2} implies the second equality of (\ref{simi1}) and the first part of \ref{f2}. Since $\nu$ satisfies the Gibbs property where $P$ is given by the topological pressure of $\{\log r_n\}_{n=1}^{\infty}$, it is an equilibrium state for $\{\log r_n\}_{n=1}^{\infty}$ and Theorem \ref{go1}\ref{v2} implies \ref{f2}.
\end{proof}

\section{Applications}\label{apli}
\subsection{Relative pressure functions and images of Gibbs measures for continuous functions}\label{a1}
We apply Theorem \ref{go1} to the sequences $\G$ associated to relative pressure functions (see (\ref{relatived})). Corollaries \ref{aplrelative} and \ref{basic0} connect $\G$ with images of Gibbs measures. 
For a survey of the study of images of Gibbs measures, see the paper by Boyle and Petersen \cite{BP}. 
Given a one-block factor map $\pi: X\rightarrow Y$ between subshifts and 
an invariant measure $\mu$ on $X$, define the image $\pi\mu\in M(Y, \sigma_Y)$ by 
$\pi\mu(B)= \mu(\pi^{-1}B)$ { for a Borel set  $B$ of $Y$}. Let $ f\in C(X)$.
For $y=(y_1, \dots, y_n, \dots)\in Y$,
define $E_n(y)$ to be a set consisting of exactly one point from each cylinder $[x_1\dots x_n]$ in $X$ such that $\pi(x_1\dots x_n)=y_1\dots y_n$. For $n\in\N$, 
\begin{equation}\label{beq}
g_n(y):=\sup_{E_n(y)}\{\sum_{x\in E_n(y)}e^{(S_nf)(x)}\}.
\end{equation}
Then 
\begin{equation*}\label{pesh}
P(\sigma_X, \pi, f)(y)=\limsup _{n\rightarrow\infty} \frac{1}{n} \log {g_n}(y),  
\mu\text{-almost everywhere for any } \mu\in M(Y,\sigma_Y)
\end{equation*}
(see \cite{PS, Fe1}). 
It is clear that $\G=\{\log g_n\}_{n=1}^{\infty}$ has bounded variation. 

 \begin{coro}\label{aplrelative}
Let $\pi: X\to Y$ be a one-block factor map between subshifts where $X$ is a subshift with the weak specification property. Given $f\in C(X)$, let $\G=\{\log g_n\}_{n=1}^{\infty}$ be  the sequence on $Y$ associated to $P(\sigma_X, \pi, f)$.
For $y\in Y$, define a function $g: X\to \R$ by 
\begin{equation}\label{defineg}
g(y)=\lim_{n\to \infty}\log \left(\frac{g_n(y)}{g_{n-1}(\sigma_Y y)}\right)
\end{equation}
if the limit exists at every point $y\in Y$.
Then the following hold. 
\begin{enumerate} [label=(\roman*)]
\item \label{almostcase}If the sequence $\G$ is almost additive on $Y$, 
then $g$ is bounded and Borel measurable on $Y$ and 
\begin{equation}\label{G1}
\int P(\sigma_X, \pi, f)d\mu
=\int g d\mu \text{ for every } \mu\in M(Y, \sigma_Y).
\end{equation}
Then 
the unique equilibrium state for $\G$ which is also the unique invariant Gibbs measure for $\G$ is the unique equilibrium state for $g$  and it is the unique  invariant Gibbs measure for $g$.
\item \label{walmostcase} If $\G$ is weakly almost additive on $Y$, then $g$ is bounded and Borel measurable on $Y$  and (\ref{G1}) holds. If,  in addition, $\G$ has an equilibrium state which is also an invariant Gibbs measure for $\G$, then it  is an equilibrium state for  $g$  which is also an invariant weak Gibbs measure for $g$.

\end{enumerate}
\end{coro}
\begin{rem}
The function $g$ is similar to  the function $u$ defined 
in \cite{PK, K}. 
\end{rem}
\begin{proof}
Set $f_n=g_n$ in Theorem \ref{go1}. 
\end{proof}

The next lemma connects relative pressure functions with images of Gibbs measures.
\begin{lemma} \label{fact1} (See \cite{Fe1,Y3})
Let $\pi: X\to Y$ be a one-block factor map between subshifts where $X$ is a subshift with the weak specification property. If $\mu$ is an invariant  Gibbs measure for $f\in C(X)$, then $\pi\mu$ is  the unique equilibrium state for the sequence 
$\G$ on $Y$ associated to $P(\sigma_X, \pi, f)$ and it is the unique invariant Gibbs measure for $\G$. 
\end{lemma}
The next corollary is useful to study the images of Gibbs measures under factor maps and applied in Section \ref{nicephd}.  
\begin{coro}\label{basic0}
Under the assumption of Corollary \ref{aplrelative}, suppose $f\in C(X)$ has an invariant Gibbs measure $\mu$ and let $g$ be defined as in Corollary \ref{aplrelative}. If $\G$ is almost additive (resp. weakly almost additive) on $Y$, then $\pi\mu$ is the unique invariant Gibbs (resp. weak Gibbs) measure for $g$ and it is the unique equilibrium state for $g$. 
\end{coro}
\begin{proof}
Corollary \ref{aplrelative} and  Lemma \ref{fact1} imply the result.
\end{proof}

\subsection{The fiber-wise sub-positive mixing  property of factor maps and almost additivity of subadditive sequences associated to relative pressure functions} \label{fiberaaR}

The fiber-wise sub-positive mixing property of factor maps is a sufficient condition for a factor map to send all Markov measures to Gibbs measures \cite{Yo}. This property is assumed in various papers in order to study images of Gibbs measures (for the most general results, see \cite{Pi2}). 

In this section, noting that the fiber-wise sub-positive mixing property is not a necessary condition for the image to be Gibbs 
(see Proposition \ref{aafiber}),
 we characterize the Gibbs property of the images of Markov measures in terms of almost additivity of the sequences associated to relative pressure functions, for  a special type of factor maps (Theorem \ref{phd1} and Corollary \ref{phd2}).  Theorem \ref{phd1} gives an answer to [Q1] from Section 1 under a particular setting (Corollary  \ref{phd2}).

\begin{defi}
\label{subp} \cite{Pi2, Yo} 
Let $\pi: X \to Y$ be a one-block factor map between subshifts where $X$ is a shift of finite type. 
Then $\pi: X\to Y$  is {\em fiber-wise sub-positive mixing}  if there exists $k\in \N$ such that  for any 
$w\in B_k(Y)$, $u_1\dots u_k \in B_k(X),  v_1\dots v_k\in B_k(X)$ satisfying
$\pi(u_1\dots u_k)=\pi(v_1\dots v_k)=w$, there exists $ a_1\dots a_k\in B_k(X)$ with 
$a_1=u_1$ and $a_k=v_k$ such that $\pi(a_1\dots a_k)=w$.
\end{defi}

\begin{thm}\cite{Pi2}\label{bestP}
Let $\pi: X\to Y$ be a one-block factor map between subshifts, where $X$ is a topologically mixing shift of finite type, satisfying the fiber-wise sub-positive mixing property. If $\mu$ is an invariant Gibbs measure for $f\in C(X)$, then  $\pi\mu$ is an invariant Gibbs measure for  $\psi \in C(Y)$ where $\psi$ is defined by
$\psi(y)=\lim_{n\to \infty}\log ({\pi\mu[y_1\dots y_n]}/{\pi\mu[y_2\dots y_n]})$.
\end{thm}

Corollary  \ref{basic0} enables us to consider conditions of images of Gibbs measures to be Gibbs for continuous functions regardless the existence of the fiber-wise sub-positive mixing property.
 \begin{prop} \label{aafiber}
Let $\pi: X\to Y$ be a one-block factor map where $X$ is a topologically mixing shift of finite type and $Y$ is a subshift.
Let $\mu$ be the unique invariant Gibbs measure for $f\in C(X)$ and $\G$ be the subadditive sequence associated to the relative pressure function $P(\sigma_X, \pi, f)$. 
\begin{enumerate}[label=(\roman*)]
\item \label{simple1}
If $\pi$ is fiber-wise sub-positive mixing, then $\G$ is almost additive on $Y$.
\item \label{simple2} 
The almost additivity of $\G$ does not imply the fiber-wise sub-positive mixing property of a factor map. There exist a one-block factor map not fiber-wise sub-positive mixing and $f\in C(X)$ such that the image $\pi\mu$ is 
an invariant Gibbs measure for some $g\in C(Y)$ where $ \G$ is almost additive on $Y$. 
\end{enumerate}
\end{prop}
\begin{proof}
If $\pi$ is fiber-wise sub-positive mixing, by Theorem  \ref {bestP}  there exists $g\in C(Y)$ for which $\pi\mu$ is Gibbs.  It is easy to see that $\F$ is almost additive because $\pi\mu$ is Gibbs for $g$ and $\G$  \cite[Corollary 6.11 (iii)]{Y4}.  
Example \ref{simpleE} in Section \ref{examplesagain} proves \ref{simple2}.
 \end{proof}
 
 In the rest of the paper, we study necessary and sufficient conditions for images of the Gibbs measures for continuous functions to be Gibbs by an explicit construction, under the setting below.\\

\noindent \textbf{[Setting (C)]}. Let $X\subseteq \{1,2,3\}^{\N}$ and $Y\subseteq\{1,2\}^{\N}$ be irreducible shifts of finite type such that $Y$ is the full shift or has the transition matrix $\begin{pmatrix}
0 & 1 \\
1 & 1 
\end{pmatrix}$. Let $\pi: X\rightarrow Y$ be a one-block factor map such that $\pi(1)=1, 
\pi(2)=\pi(3)=2$. Let $f$ be a function of two coordinates and define $f[ij]:=f(x)$ if $x\in [ij]$.
Let $\G$ be the sequence associated to the relative pressure function $P(\sigma_X, \pi, f)$. 
\\

Under Setting (C), let $A=(a_{ij})_{3\times 3}$ be the transition matrix of $X$ and define 
the $3\times 3$ nonnegative matrix $M$ with entries $M(i,j)$ given by 
$M(i,j) =e^{f[ij]}a_{ij}$ for $i,j=1,2,3$. 
Given each $b_1b_2\in B_2(Y)$, let $M_{b_1b_2}$ be the $3 \times 3$ nonnegative  matrix
  with entries  $M_{b_1b_2}(i,j)$ defined by 
 \begin{equation}
    M_{b_1b_2}(i,j) =
    \begin{cases}
      e^{f[ij]}  & \text{if  $ij\in B_2(X)$ and $\pi(ij)=b_1b_2$} \\
      0       & \text{otherwise.}
    \end{cases}
    \end{equation}
 Define the submatrix  $M_{22}\vert_{\pi^{-1}(22)}:=(M_{22}(i,j))_{i, j\in \{2,3\}}$.

\begin{thm}\label{phd1}
Under Setting (C), let $\mu$ be an invariant Gibbs measure for $f$. 
Then the following hold.
\begin{enumerate}[label=(\roman*)]
\item \label{g1}
 The sequence $\G$ is almost additive on $Y$ if and only if there exists a Borel measurable function $\hat g$ on $Y$ for which $\pi\mu$ is a unique invariant Gibbs measure. If  $\G$ is merely weakly almost additive on $Y$, then there exists a Borel measurable  function $\hat g$ on $Y$ for which $\pi\mu$ is a unique invariant weak Gibbs measure.

\item \label{revise1}If $M_{22}\vert_{\pi^{-1}(22)}\neq \begin{pmatrix}
0 & a_1 \\
a_2 & 0 
\end{pmatrix}$,  $a_1, a_2> 0$, then \ref{g1} holds for a function $\hat g\in C(Y)$.

\item \label{revise2}If $M_{22}\vert_{\pi^{-1}(22)}= \begin{pmatrix}
0 & a_1 \\
a_1 & 0 
\end{pmatrix}$, $a_1>0$, and $M$ satisfies  $M_{21}=M_{23}$ or $M_{12}=M_{13}$, then  \ref{g1} holds for a function $\hat g\in C(Y)$.
 
\end{enumerate}
\end {thm}
\begin{proof}
See the proof in Section \ref{nicephd}.
\end{proof}
\begin{rem} Let $A=(a_{ij})_{k\times k}$ be a transition matrix of an irreducible shift of finite type $X$ over $k$ symbols.  Let  
 $P=(P_{ij})$ be a stochastic matrix where 
$P_{ij}>0$ exactly when $a_{ij}>0$ and let $p=(p_1, \dots p_k)$ be the unique probability vector 
with $pP=p$. 
Then a one-step Markov measure $\mu\in M(X,\sigma_X)$ is defined by
$\mu[x_1\dots x_k]=p_1P_{12} \dots P_{k-1k}.$
If a Markov measure $\mu$ is defined by a $3 \times 3$ stochastic matrix $P=(P_{ij})$, 
set $f(x):=\log P_{i j}$ for $x\in [i,j]$. Then $\mu$ is an equilibrium state for $f$ with the Gibbs property. Hence 
Theorem \ref{phd1} is valid for any one-step Markov measure $\mu$ on $X$.
\end{rem}

\begin{coro}\label{phd2}
Under Setting (C), let $\mu$ be the measure of maximal entropy for $\sigma_X$, i.e., 
$h_{\mu}(\sigma_X)=\sup\{h_{\nu}(\sigma_X):\nu\in M(X,\sigma_X)\}$. 
 Let $\Phi$ be the sequence associated to the relative pressure function $P(\sigma_X, \pi, 0)$.  
Then $\Phi$ is almost additive on $Y$ if and only if there exists a function $\hat g \in C(Y)$ for 
which $\pi\mu$ is a unique  invariant Gibbs measure. If  $\Phi$ is merely weakly almost additive on $Y$, then there exists $\hat g\in C(Y)$ for which $\pi\mu$ is a unique invariant weak Gibbs measure. 
\end{coro}
\begin{proof}
Set $f=0$ in Setting (C) in Theorem \ref{phd1}. Then $M_{22}\vert_{\pi^{-1}(22)}= \begin{pmatrix}
0 & 1 \\
1 & 0 
\end{pmatrix}$. By Lemma \ref{exception}, the assumption of Theorem \ref{phd1} \ref{revise2}  holds. 
 

\end{proof}

\begin{coro}\label{best}
Under Setting (C),  the following hold. \begin{enumerate}[label=(\roman*)]
\item \label{rr1}
The sequence $\G$ is almost additive if and only if there exists a Borel measurable function $\hat g$ on $Y$ for which there exists a unique invariant  Gibbs measure with the following property
\begin{equation}\label{special}\lim_{n\rightarrow \infty} \frac{1}{n} 
\lVert \log g_{n}-S_n \hat g \rVert_{\infty}=0.
\end{equation} If $\G$ is merely weakly almost additive, then there exists a Borel measurable function $\hat g$ satisfying (\ref{special}) for which there is a unique invariant weak Gibbs measure.
\item \label{r1}
Under the assumption of Theorem \ref{phd1}\ref{revise1} or \ref{revise2}, \ref{rr1} holds for  $\hat g\in C(Y)$. 
\end{enumerate}
\end{coro}

\begin{proof}
Recall from Lemma \ref{fact1} that if $\mu$ is the unique Gibbs measure for $f$ then $\pi\mu$ is the unique Gibbs measure for $\G$ and it is the unique equilibrium state for $\G$.  
If $\G$ is almost additive, we take $\hat g$ from Theorem \ref{phd1}\ref{rr1}. 
(\ref{special}) follows immediately since $\pi\mu$ is Gibbs for $\G$ and 
$\hat g$.
For the reverse implication, note that $\G$ and $\hat g$ have the same equilibrium states. 
If $\nu$ is a Gibbs measure 
associated to $\hat g$, then it is an equilibrium state for $\hat g$ and hence $\pi\mu=\nu$. 
Since $\pi\mu$ is a Gibbs measure for $\G$ and $\hat g$, $\G$ is almost additive. Similar arguments yield the second statement of \ref{rr1}. \ref {r1} follows from Theorem \ref{phd1}.


\end{proof}

\section{Proof of Theorem \ref{phd1}} \label{nicephd}
This section is devoted to the proof of Theorem \ref{phd1}. Given a  factor map $\pi: X \to Y$ between shifts of finite type,
define for $n\geq 2, y\in Y$, 
\begin{eqnarray} \label{useful2}
\begin{split}
h_n(y):=&\sup\{\sum_{x\in [x_1\dots x_n], \pi(x_1\dots x_n)=y_1\dots y_n}e^{(S_{n-1}f)(x)}\}\\
\end{split}
\end{eqnarray}
where $x$ is a point  chosen from each cylinder set $[x_1\dots x_n]$.
For $n=1, y\in Y$, define $h_1(y):=g_1(y)$  where $g_1$ is defined in Corollary \ref{aplrelative} and 
let  $\H:=\{\log h_n\}_{n=1}^{\infty}$ on $Y$.
\begin{prop}\label{connectgh}
Let $\pi: X \to Y$ be a factor map between shifts of finite type 
and $\G=\{\log g_n\}_{n=1}^{\infty}$ be the sequence defined  in Corollary \ref{aplrelative}. Then the following hold.
\begin{enumerate}[label=(\roman*)]
\item \label{11}
 There exists a constant $A>0$ such that for every $n\in \N, y \in Y$,
$$\frac{1}{A}\leq \frac{g_n(y)}{h_n(y)}\leq {A}.$$
\item \label{111}
$P(\sigma_X, \pi, f)(y)=\limsup _{n\rightarrow\infty} (1/n) \log {g_n}(y) =
\limsup _{n\rightarrow\infty} (1/n) \log {h_n}(y)$, 
$\mu$-almost everywhere for any  $\mu\in M(Y,\sigma_Y)$.
\item \label{12} $\G$ is almost additive (resp. weakly almost additive) if and only if 
$\H$ is almost additive (resp. weakly almost additive). 
\item \label{13}
Define $\hat {h}(y):=\lim_{n\to \infty} h_n(y)/h_{n-1}(\sigma_Y y)$ for each $y\in Y$  if the limit exists. If $\H$ is weakly almost additive, then for every $\mu\in M(Y, \sigma_Y)$,
\begin{equation}\label{key2} 
\int P(\sigma_X, \pi, f)d\mu=\lim_{n\to \infty}\frac{1}{n}\int \log g_n d\mu=\lim_{n\to \infty} \frac{1}{n}\int \log h_nd\mu=\int \hat{h} d\mu
\end{equation}
Then there exists  $C_n>0$ such that for every $n\in \N, y \in Y$,
\begin{equation}\label{weakGh}
\frac{1}{C_n}\leq \frac{h_n(y)}{e^{(S_n\hat h)(y)}}\leq {C_n}.
\end{equation}
If $\H$ is almost additive, then there exists $C>0$ such that $C_n=C$ for all $n\in\N$.
\item \label{14} If $f=0$, then  $g_n(y)=h_n(y)$ for all $n\in \N$, that is, $\G=\H$, and $\hat {h}=g$ where 
$g$ is defined in (\ref{defineg}).
\end{enumerate}
\end{prop}
\begin{proof}
For \ref{11}, let $M=\max\{f(x): x\in X\}$ and $m=\min\{f(x): x\in X\}$
It is easy to see from the definitions of $\G$ and $\H$ that we have $e^{m}h_{n}(y)\leq  g_n(y)\leq e^{M}h_{n}(y)$ for every $n\in\N$ and $y\in Y$. 
\ref {111} and \ref {12} follow from \ref{11}. Theorem \ref{go1} combined with Corollary \ref{aplrelative} and \ref{12} yields the equalities in \ref{13}. To see (\ref{weakGh}),
recall that Lemma \ref{prego1}\ref{co2} holds by setting $C=C_n$ if (\ref{aa1}) holds by setting $C=C_n$ 
where $\lim_{n\to \infty}(1/n)C_n=0$. Similarly Lemma \ref{prego1} \ref{co1} holds by 
setting $C=C_n$ if (\ref{aa2}) holds by setting $C=C_n$ 
where $\lim_{n\to \infty}(1/n)C_n=0$. 
Now define for $y \in Y$, $\underline{\hat {h}}(y):=\liminf_{n\to \infty}\log (h_n(y)/h_{n-1}(\sigma_Y(y))$  and 
$\overline{\hat {h}}(y):=\limsup_{n\to \infty}\log (h_n(y)/h_{n-1}(\sigma_Y(y))$. By setting $\underline{f}=\underline{\hat {h}}$ and $\overline{f} =\overline{\hat {h}}$ in Lemma  \ref{prego1}, we obtain the results.
\ref{14} is obvious.  
\end{proof}
\begin{coro}\label{keycoro}
Corollary  \ref{aplrelative} \ref{almostcase} \ref{walmostcase} and Corollary \ref{basic0} hold if we replace $\G$ and $g$ by $\H$ and by $\hat h$ respectively.
\end{coro}
\begin{proof}
We apply Theorem \ref{go1} for $\H$ and $h$, noting that $\G$ and $\H$ have the same equilibrium state with 
the Gibbs property by Proposition \ref{connectgh}.
\end{proof}
We study $\H$ and find $\hat h$ explicitly  to prove Theorem \ref{phd1}. 
From now on, assume that $f$ is function of two coordinates on $X$. Then for each $n\geq 2, y\in Y$ we have 
\begin{equation} \label{useful22}
h_n(y)=\sum_{x_1\dots x_n\in B_n(X), \pi(x_1\dots x_n)=y_1\dots y_n}e^{f[x_1x_2]+f[x_2x_3]+\cdots +f[x_{n-1}x_n]}
\end{equation}
Since $h_n$ is a function of $n$ coordinates, define $h[b_1\dots b_n]:=h_n(y)$ for $y\in [b_1\dots b_n]$.
\begin{lemma} \label{thesisE}
 Let $X\subseteq \{1,\dots, l\}^{\N}, l>2$ and $Y\subseteq\{1,2\}^{\N}$ be irreducible subshifts.  Let $\pi: X\rightarrow Y$ be a one-block factor map such that $\pi(1)=1, 
\pi(i)=2$, for  $i=2,\dots, l$.  Let  $f$ be a  function on $X$ of two coordinates. 
Suppose that 
$$\lim_{n\to\infty} \frac{h[2^n]}{h[2^{n-1}]}\text{and} \lim_{n\to\infty} \frac{h[12^n]}{h[2^{n}]} \text{exist.} $$ 
 Then $\hat{h}$ defined in Proposition \ref{connectgh} \ref{13} is given by
  \begin{equation} \label{specifich}
  \hat{h}(y)=
\begin{cases}
\log h[21]   & \mbox{  $ y\in [21], $}\\
\log \big (\frac{h[2^n1]}{h[2^{n-1}1]\vert} \big)
   & \mbox{  $ y\in [2^n1], n\geq 1 $}\\
 \lim_{n\rightarrow\infty} \log \big(\frac{h[2^n]}{h[2^{n-1}]}\big) & \mbox{  $ y=2^{\infty} $}\\
  {f[11]}   & \mbox{  $ y\in [1^n2], n\geq 2,\text{ or }  y=1^{\infty} $}\\
  \log \big(\frac{h[12^n1]}{h[2^{n}1]} \big)
   & \mbox{  $ y\in [12^n1], n\geq 1 $}\\
 \lim_{n\rightarrow\infty} \log \big(\frac{h[12^n]}{h[2^{n}]} \big)& \mbox{  $y=12^{\infty} $}.
 \\
\end{cases}
 \end{equation} 
 If $h$ is continuous at $2^{\infty}$ and $21^{\infty}$, then $h$ is continuous on $Y$. 
\end{lemma}
\begin{rem}
In general $\hat h$ is not a function of finitely many coordinates. 
\end{rem}

\begin{proof}
Noting that $\pi^{-1}\{1\}=\{1\}$ 
we study the following cases.

Case1: If $y\in [21]$,  then for $n\geq 3$, it is easy to see that 
$$\frac{h_{n}(y)}{h_{n-1}(\sigma_Y y)}=\frac{h_2(y)h_{n-1}(\sigma_Y y)}{h_{n-1}(\sigma_Y y)}=h[21]$$  

Case 2: If $y\in [2^{k}1]$, for some $k\geq 2$, then for any  $n\geq k+2$, it is easy to see that 
$$\frac{h_{n}(y)}{h_{n-1}(\sigma_Y y)}=\frac{h_{k+1}(y)h_{n-k}(\sigma^k_Y y)}{h_{k}(\sigma_Yy)h_{n-k}(\sigma^k_Y y)}=\frac{h[2^k1]}{h[2^{k-1}1]}.$$

Case 3: If $y\in [12^{k}1]$, for some $k\geq 1$, then for any  $n\geq k+3$ we have
$$\frac{h_{n}(y)}{h_{n-1}(\sigma_Y y)}=\frac{h_{k+2}(y)h_{n-k-1}(\sigma^{k+1}_Y y)}{h_{k+1}(\sigma_Yy)h_{n-k-1}(\sigma^{k+1}_Y y)}=\frac{h[12^k1]}{h[2^k1]}.$$

Case 4: If $y\in [1^{k}2]$, for some $k\geq 2$, then for any  $n\geq k+3$ we obtain 
$$\frac{h_{n}(y)}{h_{n-1}(\sigma_Y y)}=\frac{h_{k}(y)}{h_{k-1}(\sigma_Yy)}=\frac{h[1^k]}{h[1^{k-1}]}=e^{f[11]}.$$
Similarly, if $y=1^{\infty}$, for $n>2$ we obtain that  $h_{n}(y)/h_{n-1}(\sigma_Yy)=e^{f[11]}.$

Case 5: If $y=12^{\infty}$, then 
$h_{n}(y)/h_{n-1}(\sigma_Yy)=h[12^{n-1}]/h[2^{n-1}]$. 

Case 6:If $y=2^{\infty}$, then 
$h_{n}(y)/h_{n-1}(\sigma_Yy)=h[2^{n}]/h[2^{n-1}]$. 

Hence we obtain the result.
\end{proof}

Now we study $\hat h$. The weak almost additivity of $\H$ plays an important role to obtain the existence of the limits in 
(\ref{specifich}) and continuity. To this end,  we apply some ideas found in \cite{CU1, Yo} and study the product of matrices associated to $\H$. Given $y\in Y$, we identify $h_n(y)$ with the sum of entries of product of matrices described below and calculate $h_n(y)$ directly.

In the rest of the section, we assume Setting (C). Recall the definitions of the matrix $M$ with entries $M(i,j)$ and the submatirx
$M_{22}\vert_{\pi^{-1}(22)}$ from the previous section. 
 We write $M_{12} =
\begin{pmatrix}
0 & z&w \\
0 & 0&0\\
 0 & 0&0
\end{pmatrix}$, for some $z,w\geq 0, (z, w)\neq (0,0)$ and  
$M_{21} =
\begin{pmatrix}
0 & 0&0 \\
\bar {x} & 0&0\\
\bar {y} & 0&0
\end{pmatrix}$, for some $\bar{x},\bar{y} \geq 0, (\bar{x}, \bar{y})\neq (0,0)$. 
 For each $b_1\dots b_n\in B_n(Y)$, $n\geq 2$, define $M_{b_1\dots b_n}:=M_{b_1b_2}
M_{b_2b_3}\dots M_{b_{n-1}b_n}$. Then  for $y\in [b_1\dots b_n]$, we have
$$ h_n(y)=\text{sum of all entries of }
 M_{b_1\dots b_n}.$$
 Since $M_{22}\vert_{\pi^{-1}(22)}$ is a nonnegative real $2 \times 2$ matrix, it has real eigenvalues and
eigenvectors  $\subset \R^2$ exist. 
 Hence $M_{22}\vert_{\pi^{-1}(22)}$ is similar to a matrix of real Jordan form $J$ and there exists an invertible matrix   
$P=\begin{pmatrix}
a & b \\
c & d
\end{pmatrix}$, $a, b, c, d\in \R$, $ad-bc\neq 0$  such that $J=P^{-1}M_{22}\vert_{\pi^{-1}(22)}P$. If a nonnegative matrix $M_{22}\vert_{\pi^{-1}(22)}$ has two distinct eigenvalues $\alpha, \beta$, then $\vert \alpha\vert > \vert \beta\vert$
implies that $\alpha>0$.
If $M_{22}\vert_{\pi^{-1}(22)}$ is similar to $J$, we write  $M_{22}\vert_{\pi^{-1}(22)} \thicksim J$. 
In the following lemmas, we continue to use the notation.  
We make straight forward arguments to obtain the lemmas, using the weak almost additivity of $\H$ and studying 
eigenvalues of real Jordan canonical forms. 

\begin{lemma}\label{2inf}
 \begin{enumerate} 
\item \label{J1} If $M_{22}\vert_{\pi^{-1}(22)} \thicksim
\begin{pmatrix}
\alpha & 1 \\
0 & \alpha 
\end{pmatrix}$, $\alpha\in \R$, 
then $\alpha>0$ and \\
$\lim_{n\to \infty} \log (h[2^{n+1}]/[2^n]) =\log \alpha$.
\item \label{dia1}
If $M_{22}\vert_{\pi^{-1}(22)} \thicksim
\begin{pmatrix}
\alpha & 0 \\
0 & \beta 
\end{pmatrix}$, where $\beta\neq -\alpha$,  
then $(\alpha, \beta)\neq (0,0)$ and \\
$\lim_{n\to \infty} \log (h[2^{n+1}]/[2^n])$ exists.
\end{enumerate} 
\end{lemma}
\begin{proof}
We first show (\ref{J1}). Let $n\geq 4$. 
Since $(M_{22}\vert_{\pi^{-1}(22)})^n=P\begin{pmatrix}
\alpha^n & n\alpha^{n-1} \\
0 & \alpha^n 
\end{pmatrix}P^{-1}$, for an invertible matrix $P$,
a straight forward calculation shows that 
\begin{eqnarray}\label{2nJ}
\begin{split}
h[2^{n+2}]&=\text{sum of all entries of the matrix $M_{22}^{n+1}$}\\
&=\frac{2(ad -bc)\alpha^{n+1}+(n+1)(a^2-c^2)\alpha^{n}}{ad-bc}>0.
\end{split}
\end{eqnarray}
Note that $\alpha\neq 0$ since $2^{n+2}\in B_{n+2}(Y)$. 

Case 1. $a^2-c^2\neq 0$. We obtain $\lim_{n\to \infty}h[2^{n+2}]/h[2^{n+1}]=\alpha$.  Hence $\alpha>0$.

Case 2. $a^2-c^2=0$. Since $ad-bc\neq 0$, we obtain $\lim_{n\to \infty}h[2^{n+2}]/h[2^{n+1}]=\alpha$. Hence $\alpha>0$.

Next we show (\ref{dia1}). Since
$(M_{22}\vert_{\pi^{-1}(22)})^n=P\begin{pmatrix}
\alpha^n & 0 \\
0 & \beta^n 
\end{pmatrix}P^{-1}$
for an invertible matrix 
$P$, 
simple calculations show that 
\begin{eqnarray}\label{2nD}
\begin{split}
h[2^{n+2}]=&\text{sum of all  entries of $M_{22}^{n+1}$}\\
&=\frac{(a+c)(d-b)\alpha^{n+1}+(a-c)(d+b)\beta^{n+1}}{ad-bc}>0.
\end{split}
\end{eqnarray}
Case 1. $\alpha$ or  $\beta$ is zero. Note that $(\alpha, \beta)\neq (0,0)$. Without loss of generality, assume $\beta=0$. Then $(a+c)(d-b)\neq 0$ and 
$\lim_{n\to \infty}h[2^{n+2}]/h[2^{n+1}]=\alpha$.  Hence $\alpha>0$.

\noindent Case 2. $\alpha\neq 0$ and $\beta\neq 0$

(1) $\vert \alpha\vert =\vert \beta\vert $. 
\begin{enumerate} [label=(\roman*)]
\item $\alpha = \beta$. Then $\lim_{n\to \infty}h[2^{n+2}]/h[2^{n+1}]=\alpha=\beta.$ Hence $\alpha=\beta>0$.



\end{enumerate}

(2) $\vert \alpha \vert \neq \vert \beta \vert $. 
The Jordan canonical form is unique up to rearranging the Jordan blocks.
Hence without loss of generality, assume that $\vert \alpha\vert>\vert \beta\vert$.
\begin{enumerate} [label=(\roman*)]
\item $(a+c)(d-b)\neq 0.$ \text{Then} $\lim_{n\to \infty}h[2^{n+2}]/h[2^{n+1}]=\alpha$.  Hence $\alpha>0$.
 \item $(a+c)(d-b)= 0.$ \text{Then} $(a-c)(d+b)\neq 0$ \text{and} $\lim_{n\to \infty}h[2^{n+2}]/h[2^{n+1}]=\beta.$ Hence $\beta>0$.
 \end{enumerate}
 \end{proof}

\begin{lemma}\label{12inf}
Suppose $\H=\{\log h_n\}_{n=1}^{\infty}$ is weakly almost additive.
\begin{enumerate} 
\item \label{J2} If $M_{22}\vert_{\pi^{-1}(22)}\thicksim
\begin{pmatrix}
\alpha & 1 \\
0 & \alpha 
\end{pmatrix}$,
then $\lim_{n\to \infty} \log (h[12^{n+1}]/[2^{n+1]})$ exists.
\item \label{dia2}
If $M_{22}\vert_{\pi^{-1}(22)} \thicksim
\begin{pmatrix}
\alpha & 0 \\
0 & \beta 
\end{pmatrix}$,  where $\beta\neq -\alpha$,
then $\lim_{n\to \infty} \log (h[12^{n+1}]/[2^{n+1}])$ exists. 
\end{enumerate} 

\end{lemma}
\begin{proof}
We show (\ref{J2}). 
Using a real Jordan form of $M_{22}\vert_{\pi^{-1}(22)} $ as in the proof of Lemma \ref{2inf}, we obtain 
 \begin{eqnarray} \label{J12n}
\begin{split}
h[12^{n+2}]&=\text{sum of all entries of the matrix $M_{12}M_{22}^{n+1}$}\\
&=\frac{(ad-bc)(z+w)\alpha^{n+1}+(n+1)(az+cw)(a-c)\alpha^{n}}{ad-bc},
\end{split}
\end{eqnarray}
We use  (\ref{2nJ}) and (\ref{J12n}), 

 Case 1. $(az+cw)(a-c)\neq 0$. 
By (\ref{2nJ}), we note  that $a+c= 0$ gives a contradiction to
the weak almost additivity of $\H$. Hence $a+c\neq 0$. We obtain $\lim_{n\to \infty}h[12^{n+2}]/h[2^{n+2}]=(az+cw)/(a+c).$

Case 2. $(az+cw)(a-c)= 0$.
By (\ref{2nJ}), we note  that $a+c= 0$ by the weak almost additivity 
of $\H$. Then $z+w\neq 0$ and $\lim_{n\to \infty}h[12^{n+1}]/h[2^{n}]=(z+w)/2$.



 Next we show (\ref{dia2}).
 Using a real Jordan form of $M_{22}\vert_{\pi^{-1}(22)}$ as in the proof of Lemma \ref{2inf},
  \begin{eqnarray}\label{dia121}
\begin{split}
h[12^{n+2}]&=\text{sum of all entries of the matrix $M_{12}M_{22}^{n+1}$}\\
&=\frac{\alpha^{n+1}(d-b)(az+cw)+\beta^{n+1} (a-c)(bz+dw)}{ad-bc}.
\end{split}
\end{eqnarray} 
We apply (\ref{2nD}) and (\ref{dia121}). Note that $h[12^{n+2}]\neq 0$ and $h[2^{n+2}]\neq 0$ because $12^{n+2}$, $2^{n+2}\in B(Y)$. 

Case 1. $\beta=0$. Then $(a+c)(d-b)\neq 0$ and $(d-b)(az+cw)\neq 0$. Then
$\lim_{n\to \infty}h[12^{n+1}]/h[2^{n}]=(az+cw)/(a+c)$.

Case 2. $\alpha\neq 0$ and $\beta\neq 0$.

(1) $\alpha  =\beta $. 
By a direct computation with $ad-bc\neq 0$, $\lim_{n\to \infty}h[12^{n+1}]/h[2^{n}]=(z+w)/2$.

(2) $\vert \alpha \vert \neq  \vert \beta \vert $. Assume that $ \vert \alpha\vert >\vert  \beta \vert $.
\begin{enumerate} [label=(\roman*)]
\item $(d-b)(az+cw)\neq 0.$ We note that the weak almost additivity of $
\H$ implies that $(d-b)(az+cw)\neq 0 \iff (a+c)(d-b)\neq 0.$
Hence $\lim_{n\to \infty}h[12^{n+1}]/h[2^{n}]=(az+cw)/(a+c)$.
 \item $(d-b)(az+cw)=0.$  Then $(d-b)(az+cw)=0 \iff (a+c)(d-b)=0.$ Thus $\lim_{n\to \infty}h[12^{n+1}]/h[2^{n}]=(bz+dw)/(b+d).$
 \end{enumerate} 
 \end{proof}
 
 
 \begin{lemma}\label{conti2}
 Suppose $\H=\{\log h_n\}_{n=1}^{\infty}$ is weakly almost additive. 
If $M_{22}\vert_{\pi^{-1}(22)} \nsim
\begin{pmatrix}
\alpha & 0 \\
0 & -\alpha
\end{pmatrix}$,  $\alpha\neq 0$, 
 then $\hat h$ is continuous at $2^{\infty}\in Y$.
 \end{lemma}
\begin{proof}
We apply the proof of Lemma \ref{2inf}. By the definition of $\hat h$, 
it is enough to show that 
\begin{equation}\label{hconat2}
\lim_{n\to \infty} \log \frac{h[2^{n+2}]}{h[2^{n+1}]}=
\lim_{n\to \infty} \log \frac{h[2^{n+2}1]}{h[2^{n+1}1]}.
\end{equation}
Suppose
$M_{22}\vert_{\pi^{-1}(22)} \thicksim
\begin{pmatrix}
\alpha & 1 \\
0 & \alpha 
\end{pmatrix}$.
Then we obtain that
 \begin{eqnarray} \label{J21n}
\begin{split}
h[2^{n+2}1]&=\text{sum of all  entries of the matrix $M_{21}M_{22}^{n+1}$}\\
&=\frac{(ad -bc)(\bar{x}+\bar{y})\alpha^{n+1}+(n+1)(a+c)(-c\bar{x}+a\bar{y})\alpha^{n}}{ad-bc}
\end{split}
\end{eqnarray}

 Case 1. $(a+c)(-c\bar{x}+a\bar{y})\neq 0$. Then $\lim_{n\to \infty}h[2^{n+2}1]/h[2^{n+1}1]=\alpha$.

 
Case 2. $(a+c)(-c\bar{x}+a\bar{y})= 0$.  Then $(ad -bc)(\bar{x}+\bar{y})\neq 0$ and we obtain that  $\lim_{n\to \infty}h[2^{n+2}1]/h[2^{n+1}1]=\alpha$.

Next we next show (\ref{hconat2}) when $ M_{22}\vert_{\pi^{-1}(22)}\thicksim
\begin{pmatrix}
\alpha & 0 \\
0 & \beta 
\end{pmatrix}$. By simple calculations  \begin{equation} \label{dia21n}
h[2^{n+2}1]=\frac{(d\bar{x}-b\bar{y})(a+c)\alpha^{n+1}+(a\bar{y}-c\bar{x})(b+d)\beta^{n+1}}
{ad-bc}.
\end{equation}
Case 1. $\beta=0$. Then $\alpha\neq 0$. Since $h[2^{n+2}1]\neq 0$ for any $n\in\N$, we obtain
$\lim_{n\to \infty}h[2^{n+1}1]/h[2^{n}1]=\alpha$. Hence (\ref{hconat2}) holds. 

Case 2. $\alpha\neq 0$ and $\beta\neq 0$.

 (1) $\alpha  =\beta $. 
Since $h[2^{n+2}1]\neq 0$ for any $n\in\N$,  
$\lim_{n\to \infty}h[2^{n+1}1]/h[2^{n}1]=\alpha=\beta$.
Hence (\ref{hconat2}) holds.

(2) $\vert \alpha \vert \neq \vert \beta \vert $. Assume that $\vert \alpha\vert >  \vert \beta\vert$. 
\begin{enumerate} [label=(\roman*)]
\item  $(d\bar{x}-b\bar{y})(a+c)= 0.$ \text{Then} 
$(a\bar{y}-c\bar{x})(b+d)\neq 0$ and 
$ \lim_{n\to \infty}h[2^{n+1}1]/h[2^{n}1]=\beta.$ We claim that $(a+c)(d-b)=0$.
Suppose not. 
Then by (\ref{2nD}) and 
(\ref{dia21n}), we obtain $ \lim_{n\to \infty}h[2^{n+1}1]/h[2^{n}]=0$, which is a contradiction to the weak almost additivity of $\H$. 
By the proof of Lemma \ref{2inf}  (\ref{dia1}) Case 2 (2), (\ref{hconat2}) holds.
\item $(d\bar{x}-b\bar{y})(a+c)\neq 0.$ \text{Then} 
$\lim_{n\to \infty}h[2^{n+1}1]/h[2^{n}1]=\alpha$. 
Suppose that $(a+c)(d-b)= 0$. Then  by (\ref{2nD}) and 
(\ref{dia21n}), we obtain $\lim_{n\to \infty}h[2^{n+1}1]/h[2^{n}]=\infty$, which is a contradiction to the weak almost additivity of $\H$.  Hence  $(a+c)(d-b)\neq 0$ and the proof of Lemma \ref{2inf}  (\ref{dia1}) Case 2 (2) implies (\ref{hconat2}). 
\end{enumerate}
\end{proof}

\begin{lemma}\label{conti12}
Suppose $\H=\{\log h_n\}_{n=1}^{\infty}$ is weakly almost additive.
\begin{enumerate} 
\item \label{J4} If $M_{22}\vert_{\pi^{-1}(22)} \thicksim
\begin{pmatrix}
\alpha & 1 \\
0 & \alpha 
\end{pmatrix}$, then $\alpha>0$ and the following hold.
\begin{enumerate}
\item
If $(wc+az)(-c\bar x+a\bar y)\neq 0$, then $\hat h$ is continuous at $12^{\infty}\in Y$. 
\item If $(wc+az)(-c\bar x+a\bar y)= 0$, then $\hat h$ is continuous at $12^{\infty}\in Y$ if and only if $(z+w)/2=(\bar x w+\bar y w)/(\bar x+\bar y)$.
\end{enumerate}

\item \label{dia4}
If $M_{22}\vert_{\pi^{-1}(22)} \thicksim
\begin{pmatrix}
\alpha & 0 \\
0 & \beta 
\end{pmatrix}$, then $(\alpha, \beta)\neq (0,0)$ and  the following hold.
\begin{enumerate}
\item  If $ \alpha \neq \beta$, where $\beta\neq -\alpha$, then $\hat h$ is continuous at $12^{\infty}\in Y$.
\item If $ \alpha  = \beta  $, then $\hat h$ is continuous at $12^{\infty}\in Y$ if and only if
$(\bar{x}-\bar{y})(w-z)=0$. 
\end{enumerate}
\end{enumerate}
\end{lemma}

\begin{proof}
We apply the proof of Lemma \ref{12inf}. It is enough to show that 
\begin{equation}\label{hcon12}
\lim_{n\to \infty} \log \frac{h[12^{n+2}1]}{h[2^{n+2}1]}=
\lim_{n\to \infty} \log \frac{h[12^{n+2}]}{h[2^{n+2}]}.
\end{equation}

We show (\ref{J4}). Some calculations show that 
\begin{eqnarray}\label{12nJ}
\begin{split}
h[12^{n+2}1]&=\text{sum of all entries of the matrix $M_{12}M_{22}^{n+1}M_{21}$}\\
&=\frac{(ad-bc)(\bar x z+\bar y w)
\alpha^{n+1}+(n+1)\alpha^{n}(wc+az)(-c\bar x+a \bar y)}{ad-bc}
\end{split}
\end{eqnarray}
We study the following cases by using (\ref{2nJ}) and (\ref{J21n}). 

Case 1. $(wc+az)(-c\bar x+a\bar y)\neq 0$. Suppose that 
$a+c=0.$ Then ${h[12^{n+2}1]}/{h[2^{n+2}1]}$ is not bounded and it is a contradiction to the weak almost additivity of $\H$. Hence $a+c\neq 0$ and  
$ \lim_{n\to \infty}h[12^{n+2}1]/h[2^{n+2}1]=(wc+az)/(a+c).$ In (\ref{2nJ}), if $a-c=0$, then 
$\lim_{n\to \infty} h[12^{n+2}1]/h[2^{n+2}]=\infty$. Hence $a-c\neq 0$.
By the proof of Lemma \ref{12inf} (\ref{J2}) Case 1, $\hat h$ is continuous at $12^{\infty}$. 

Case 2. $(wc+az)(-c\bar x+a \bar y)= 0$. Then $(ad-bc)(\bar x z+\bar y w)\neq 0$. 
\begin{enumerate}
\item $wc+az=0$ and  $-c\bar x+a \bar y\neq 0$. Then 
(\ref{J21n}) and the almost additivity of $\H$  imply $a+c=0$.
$\lim_{n\to \infty} h[12^{n+2}1]/h[2^{n+2}1]=(\bar x z+\bar y w)/(\bar x+\bar y)$. 
\item $wc+az=0$ and  $-c\bar x+a \bar y= 0$.
Then (\ref{J21n}) implies that $\bar x+\bar y\neq 0$ and $\lim_{n\to \infty} h[12^{n+2}1]/h[2^{n+2}1]=(\bar x z+\bar y w)/(\bar x+\bar y)$ 
\item $wc+az\neq 0$ and  $-c\bar x+a \bar y= 0$.
Then (\ref{J12n}) implies that $a-c=0$. Then $\lim_{n\to \infty}h[12^{n+2}1]/h[2^{n+2}1]=(\bar x z+\bar y w)/(\bar x+\bar y)$. 
\end{enumerate}
Hence in this case, by the result of Case 2 of the proof of Lemma   \ref{12inf} (\ref{J2})  $\hat h$ is continuous  at $12^{\infty}$ if and only if $(z+w)/2=(\bar x z+\bar y w)/(\bar x+\bar y).$

 Next we show (\ref{dia4}). 
 We have 
\begin{eqnarray}\label{12ndia}
\begin{split}
h[12^{n+2}1]&=\text{sum of all entries of the matrix $M_{12}M_{22}^{n+1}M_{21}$}\\
&=\frac{(az+cw)(d\bar{x}-b\bar{y})\alpha^{n+1} +(a\bar{y}-c\bar{x})(bz+dw)\beta^{n+1}}{ad-bc}
 \end{split}
 \end{eqnarray}
 Case 1. $\beta=0$. Then $\alpha\neq 0$ and,  by (\ref{dia21n}), $(a+c)(d\bar{x}-b\bar{y})\neq 0$.
 Then we have 
$\lim_{n\to \infty}h[12^{n+2}1]/h[2^{n+2}1]=(az+cw)/(a+c)$ and (\ref{hcon12}) holds.

Case 2. $\alpha\neq 0$ and $\beta\neq 0$.

(1) $\alpha = \beta $. 
A simple computation shows that 
$\lim_{n\to \infty}h[12^{n+1}1]/h[2^{n+1}1]=(\bar{x}z+\bar{y}w)/(\bar{x}+\bar{y})$. Hence 
$h$ is continuous at $12^{\infty}$ if and only if $(z+w)/2=(\bar{x}z+\bar{y}w)/(\bar{x}+\bar{y})$, that is, $(\bar{x}-\bar{y})(w-z)=0$. 

(2) $\vert \alpha\vert   \neq \vert \beta \vert$. Assume that $\vert \alpha \vert > \vert \beta\vert $. We use the weak almost additivity of $\H$ as in the proof of Lemma \ref{conti2}. We have 
\begin{equation}
\frac{h[12^{n+2}1]}{h[2^{n+2}1]}=\frac{(az+cw)(d\bar{x}-b\bar{y})\alpha^{n+1} +(a\bar{y}-c\bar{x})(bz+dw)\beta^{n+1}}{
(d\bar{x}-b\bar{y})(a+c)\alpha^{n+1}+(a\bar{y}-c\bar{x})(b+d)\beta^{n+1}}
 \end{equation}

\begin{enumerate} [label=(\roman*)]
\item $(az+cw)(d\bar{x}-b\bar{y})=0.$ \text{Then} $(a\bar{y}-c\bar{x})(bz+dw)\neq 0.$ By the weak almost additivity of $\H$ we have $(d\bar{x}-b\bar{y})(a+c)=0$ and $h[12^{n+2}1]/h[2^{n+2}1]=(bz+dw)/(b+d)$. 
Suppose $(a+c)(d-b)\neq0$. Then 
by (\ref{2nD}) $\lim_{n\to \infty}h[12^{n+2}1]/h[2^{n+2}]$\\$=0$, which is a contradiction to the   weak almost additivity of $\H$. 
By the proof of Lemma \ref{12inf} (\ref{dia2}), 
 $h$ is continuous at $12^{\infty}$.

\item $(az+cw)(d\bar{x}-b\bar{y})\neq 0.$ \text{Then} by the weak almost additivity of $\H$,
$(d\bar{x}-b\bar{y})(a+c)\neq0$ and $\lim_{n\to \infty}h[12^{n+2}1]/h[2^{n+2}1]=(az+cw)/(a+c)$.  
Suppose $(a+c)(d-b)=0$. Then by (\ref{2nD}) $\lim_{n\to \infty}h[12^{n+1}1]/h[2^{n+1}]=\infty$, which is a contradiction. Hence $(a+c)(d-b)\neq0$. Now we apply the proof of Lemma \ref{12inf} (\ref{dia2}).  
\end{enumerate}
\end{proof}

\begin{lemma}\label{lastlemma}
Suppose $\H=\{\log h_n\}_{n=1}^{\infty}$ is weakly almost additive and the matrix $M$ has one of the following 
forms. 
\begin{enumerate}[label=(\roman*)]
\item   \label{lastp2} $M_{22}\vert_{\pi^{-1}(22)} \thicksim
\begin{pmatrix}
\alpha & 1 \\
0 & \alpha 
\end{pmatrix}$ and 
$(wc+az)(-c\bar x+a\bar y)= 0$.

\item \label{lastp1}
$M\vert_{\pi^{-1}(22)} =
\begin{pmatrix}
\alpha & 0 \\
0 & \alpha 
\end{pmatrix}$.
\end{enumerate}
 Then define for $y\in Y$
\begin{equation} \label{specifichR}
   \hat{h}_1(y)=
\begin{cases}
\log \big(\frac{\bar{x}z+\bar{y}w}{\bar{x}+\bar{y}}\big)     & \mbox{  $ y=12^{\infty}, $}\\
\hat {h}(y)    & \mbox{  \text{ otherwise .}}
 \end{cases}
 \end{equation}
Then $\hat{h}_1\in C(Y)$. A measure $\mu$ is an invariant Gibbs 
(resp. weak Gibbs) measure 
for $\hat h_1$ if and only if $\mu$ is an invariant Gibbs 
(resp. weak Gibbs) measure for $\hat h$. 
\end {lemma}
\begin{proof}
Note that $\alpha>0$. By Lemmas \ref{2inf}, \ref{12inf} and  \ref{conti2}, $\hat h$ is defined on $Y$ and possibly has a discontinuity at $12^{\infty}$.  If $M\vert_{\pi^{-1}(22)}$ satisfies \ref{lastp2} or \ref{lastp1}, we obtain that 
$\lim_{n\to \infty}h[12^{n+1}1]/h[2^{n+1}1]=(\bar{x}z+\bar{y}w)/(\bar{x}+\bar{y})$ (see the proof of Lemma \ref{conti12}). 
Hence $\hat h_1$ is continuous at $12^{\infty}$. 
Let $y=(y_1, \dots, y_n\dots)\in Y$.
Suppose there is no $p\in \N$ such that $\sigma_Y^p(y)=12^{\infty}$. Then for every $n\in\N$, 
$e^{(S_n{\hat h})(y)}/e^{(S_n{\hat {h}_1})(y)}=1$. 
Suppose there exists $p\in \N$ such that $\sigma_Y^p(y)=12^{\infty}$. For $n\leq p$, we have $e^{(S_n{\hat h})(y)}/e^{(S_n{\hat h})(y)}=1$.  For $n\geq p+1$, if \ref{lastp2} or \ref{lastp1} occurs, then 
 the proof of Lemma \ref{12inf} implies 
$\hat h (12^{\infty})=(z+w)/2$. Hence 
$$\frac{e^{(S_n{\hat h})(y)}}{e^{(S_n {\hat h}_1)(y)}}
=\frac{(z+w)(\bar x+\bar y)}{2( \bar x z+\bar y w)}.$$
The results follow  by the definitions of the (weak) Gibbs measures. 
\end{proof}
We are left with one more case to study. 

\begin{lemma}\label{exception}
Suppose that $M_{22}\vert_{\pi^{-1}(22)} \thicksim
\begin{pmatrix}
\alpha & 0 \\
0 & -\alpha
\end{pmatrix}$ for some $\alpha$. Then $\alpha\neq 0$.
If $M_{23}=M_{32}$, then $\hat h$ is continuous on $Y$ except at the point $12^{\infty}$.
If, in addition $\bar x=\bar y$ or $z=w$, then $\hat h$ is continuous on $Y$.
In particular, if $f=0$, then $\hat h\in C(Y)$ and $\hat h=g$ where $g$ is defined in (\ref{defineg}). 
 \end{lemma}
 \begin{rem}
 We can study  ${q} (y):=\lim_{n\to \infty} \log (h_n(y)/h_{n-2}({\sigma_Y}^2(y)))^{1/2}$, $y \neq 12^{\infty}$ 
 in Theorem \ref{go1}. 
 Then $q$ is continuous at $2^{\infty}$ but the limit does not exist at $12^{\infty}$.
 \end{rem}
\begin{proof}
Clearly $\alpha\neq 0$. Since $M_{22}\vert_{\pi^{-1}(22)}$ has two distinct eigenvalues $\alpha$ and -$\alpha$, we obtain 
 $M_{22}\vert_{\pi^{-1}(22)}=\begin{pmatrix}
0 & a_1 \\
a_2 & 0
\end{pmatrix}$ for some $a_1, a_2>0$. 
For $k\in\N$,
$h[2^{2k+1}]=2{a_1}^k{a_2}^k$ and 
$h[2^{2k+2}]={a_1}^k{a_2}^k(a_1+a_2).$
Hence $h[2^{2k+1}]/h[2^{2k}]=2a_1a_2/(a_1+a_2)$ and 
$h[2^{2k+2}]/h[2^{2k+1}]=(a_1+a_2)/2$.  Thus $\hat h$ is defined at $2^{\infty}$ if 
$a_1=a_2$. By a straight forward calculation, 
$\lim_{k\to \infty} h[2^{2k+1}1]/h[2^{2k}1]=a_1a_2(\bar x+\bar y)/(a_2 \bar x+ a_1 \bar y)$ and $\lim_{k\to \infty} h[2^{2k+2}1]/h[2^{2k+1}1]=(a_2\bar x+a_1 \bar y)/(\bar x+\bar y)$. Hence $a_1=a_2$ implies the continuity at $2^{\infty}$. Similarly, we obtain
 $\lim_{k\to \infty} h[12^{2k+1}]/h[2^{2k+1}]=(z+w)/2$ and $\lim_{k\to \infty} h[12^{2k+2}]/h[2^{2k+2}]=(za_1+wa_2)/(a_1+a_2)$ and  $\hat h$ is defined at $12^{\infty}$ if 
$a_1=a_2$ or $z=w$.  Since $\lim_{k\to \infty} h[12^{2k+1}1]/h[2^{2k+1}1]=(z\bar x+w\bar y)/(\bar x+\bar y)$ and 
$\lim_{k\to \infty} h[12^{2k}1]/h[2^{2k}1]=(\bar yza_1+\bar x wa_2)/ (a_1 \bar y+a_2 \bar x)$ (by direct computations),
 $\hat h$ is continuous at $12^{\infty}$ if  $a_2{\bar x}^2=a_3{\bar y}^2$ or $z=w$.  
 If $f=0$, then the matrix $M$=transition matrix of $X$ and $a_1=a_2=1$.  Suppose neither $\bar x=\bar y$ nor $z=w$.
 Then we have a contradiction to the irreducibility of a shift of finite type $X$. 
Hence $\hat h\in C(Y)$. If $f=0$, then $\G=\H$ and $\hat h=g$.
\end{proof}

 \begin{lemma}\label{exceptionF}
For the submatrix $M_{22}\vert_{\pi^{-1}(22)} =
\begin{pmatrix}
0 & a_1 \\
a_2 &  0
\end{pmatrix}$, $a_1, a_2>0$, 
define for $y\in Y$
\begin{equation} \label{specifichR}
  \hat{h}_2(y):=
\begin{cases}
\frac{1}{2}\log \big(a_1a_2)    & \mbox{  $ y=12^{\infty}, y=2^{\infty}$}\\
\hat {h}(y)    & \mbox{  \text{ otherwise .}}
 \end{cases}
 \end{equation}
Then $\hat{h}_2$ is 
Borel measurable on $Y$ and there exists ${C}_n>0$, $\lim_{n\to \infty} (1/n)\log {C}_n=0$ such that
\begin{equation}\label{modify}
e^{-{C}_n}\leq h_n(y)/e^{(S_n \hat{h}_2)(y)}\leq e^{{C}_n} \text { for every } y\in Y, n\in\N.
\end{equation}
The unique invariant Gibbs 
(resp. weak Gibbs) measure for $\H$ on $Y$  is the unique invariant Gibbs 
(resp. weak Gibbs) measure for $\hat{h}_2$ on $Y$. 
\end{lemma}
\begin{proof}
Let 
$P:=\{y\in Y: \text { there exists } p\in \N \cup\{0\} \text { such that } \sigma^p_Yy=2^{\infty}\}$.  By Proposition \ref{connectgh}
\ref{13}, using the definitions of $\hat h$ and $\hat{h}_2$, 
(\ref{modify}) holds  for $y\notin P$. It is enough to show that 
we can find $C'_{n}>0$ such that  $\lim_{n\to \infty} (1/n)\log C'_{n}=0$ satisfying
\begin{equation}\label{modi1}
e^{-C'_{n}}\leq h_n(y)/e^{(S_n \hat{h}_2)(y)}\leq e^{C'_{n}}   \text { for } y\in P, \text {every } n\in\N.
\end{equation} 
Case 1. $y=2^{\infty}$. For $k\in \N$, by the proof of Lemma \ref{exception}
we obtain $h_{2k}(y)/e^{(S_{2k}\hat{h}_2)(y)}=(a_1+a_2)/a_1a_2$ and 
$h_{2k+1}(y)/e^{(S_{2k+1}\hat{h}_2)(y)}=2.$

Case 2
$y=12^{\infty}$ or $y=y_1\dots y_{p-1} 12^{\infty}$ for some $p\geq 2$. We consider for the second case. We can study the first case by using a similar proof.
We obtain
\begin{eqnarray*}
\begin{split}
&\frac{h_n(y)}{e^{(S_n \hat{h}_2)(y)}}=\frac{h[y_1\dots y_{p-1}12^{n-p}]}{e^{(S_n\hat{h}_2)(y)}}=
\frac{h[y_1\dots y_{p-1}1]h[12^{n-p}]}{e^{(S_{p-1}\hat{h}_2)(y)}
e^{\hat{h}_2(12^{\infty})}e^{(S_{n-p}\hat{h}_2)(2^{\infty})}}\\
\leq 
&\frac{h[y_1\dots y_{p-1}]e^{A_1}h[12^{n-p}]}{e^{S_{p-1}\hat{h}(y)}
e^{\hat{h}_2(12^{\infty})}e^{S_{n-p}\hat{h}_2(2^{\infty})}}\leq \frac{e^{C_{p-1}+A_1}h[1]h[12^{n-p}]}{(a_1a_2)^{\frac{1}{2}}(a_1a_2)^{\frac{n-p}{2}}},
\end{split}
\end{eqnarray*}
where $A_1>0$ is obtained from the weak almost additivity of $\H$ and  $C_{p-1}$ is defined in (\ref{weakGh}). 
Similarly,
\begin{eqnarray*}
\begin{split}
&\frac{h_n(y)}{e^{(S_n \hat{h}_2)(y)}}=\frac{h[y_1\dots y_{p-1}12^{n-p}]}{e^{S_n\hat{h}_2(y)}}
\geq \frac{e^{-C_{p-1}-A_1}h[1]h[12^{n-p}]}{(a_1a_2)^{\frac{1}{2}}(a_1a_2)^{\frac{n-p}{2}}}
\end{split}
\end{eqnarray*}

\begin{enumerate} [label=(\roman*)]
\item $n-p=2k, k\in \N.$
A simple calculation shows that 
$h[12^{n-p}]/(a_1a_2)^{\frac{n-p}{2}}=(z/a_2)+(w/a_1)$.
\item $n-p=2k+1, k\in \N.$
Then
$h[12^{n-p}]/(a_1a_2)^{\frac{n-p}{2}}=z+w.$
\end{enumerate}
By Case 1 and Case 2, the results follow.
\end{proof}

 \noindent \textbf{[Proof of Theorem \ref{phd1}]}
 We show \ref{g1}. If $\pi\mu$ is Gibbs for $\hat g$ and $\G$, then it is clear that the sequence $\G$ is almost additive. 
For the reverse implication, first note that $\H$ is almost additive by Proposition \ref{connectgh}.
If $M_{22}\vert_{\pi^{-1}(22)} \neq
\begin{pmatrix}
0 & a_1 \\
a_2 & 0
\end{pmatrix}$,  $a_1, a_2>0$, then
Lemmas \ref{2inf} and \ref{12inf} imply that $\hat h$ is defined on $Y$ and Borel measurable. 
In this case set $\hat g=\hat h$.
Otherwise, 
we apply Lemma \ref{exceptionF} and set $\hat g={\hat h}_2.$  
By Corollary \ref {keycoro}, we obtain the first statement. The second statement follows similarly.
Next we show \ref{revise1}. If $M\vert_{\pi^{-1}(22)}$ satisfies neither Lemma \ref{lastlemma} \ref{lastp2} nor  \ref{lastp1}, then we can set  $\hat g=\hat h$ since Lemmas \ref{conti2} and \ref{conti12} imply that $\hat h\in C(Y)$.
If $M\vert_{\pi^{-1}(22)}$ satisfies one of the conditions in Lemma \ref{lastlemma}, 
we apply Lemma \ref{lastlemma} and set $\hat g= \hat {h}_1$.  
Lemma \ref{exception} proves \ref{revise2}.

\section{Examples}\label{examplesagain}
In this section we give examples that illustrate Theorem \ref{go1} (Corollary \ref{aplrelative}) and Theorem \ref{phd1}. Example \ref{simpleE}  proves Proposition \ref{aafiber} \ref{simple2}.  See  Section \ref{nicephd} for the notation.
\begin{ex} \label{simpleE}
Given the factor map $\pi:X\to Y$ between subshifts from 
 \cite[Example 3.2]{S1},  we study the subadditive sequence $\G$ associated to the relative pressure function $P(\sigma_X, \pi, 0)$. Let $X\subset \{1,2,3\}^{\N}$ be the topologically mixing shift of finite type with the transition matrix $A$ given by 
\begin{equation*}
A=\left(
\begin{array}{ccc}
0 & 1 & 1 \\
1 & 1 & 0 \\
1 & 0 & 1
\end{array} \right).
\end{equation*}
and $\pi: X\to Y$ be the one-block factor map determined by $\pi(1)=1$ and $\pi(2)=\pi(3)=2$. Then $Y$ is the topologically mixing shift of finite type with the transition matrix $B$ where 
\begin{equation*}
B=\left(
\begin{array}{ccc}
0 & 1  \\
1 & 1  \\
\end{array} \right).
\end{equation*}
Then $\pi$ is not fiber-wise sub-positive mixing because $23\notin B_2(X)$. 
We consider the image of the Gibbs measure for $f=0$ on $X$. By Proposition \ref{connectgh} $\G$=$\H$. Then 
$$h[12^n1]=h[2^n]=h[2^n1]=h[12^n]=2 \text{ for each }n\in \N,$$   and $\H$ is almost additive. By Theorem \ref{phd1} there exists a continuous function for which $\pi\mu$ is Gibbs for a continuous function and it is given by $\hat h(=g \text{ in } (\ref{defineg}))$. 
\end{ex}

 \begin{ex}\label{ex3}
 \cite[Example 5.1]{Y1} 
 Let $X\subset \{1,2,3\}^{\N}$ be the topologically mixing shift of finite type with the transition matrix $A$ given by 
\begin{equation*}
A=\left(
\begin{array}{ccc}
0 & 1 & 0 \\
1 & 1 & 1 \\
1 & 0 & 1
\end{array} \right)
\end{equation*}
and $\pi: X\to Y$ be the one-block factor map determined by $\pi(1)=1$ and $\pi(2)=\pi(3)=2$. Then $Y$ is the topologically mixing shift of finite type with the transition matrix $B$ where 
\begin{equation*}
B=\left(
\begin{array}{ccc}
0 & 1  \\
1 & 1  \\
\end{array} \right).
\end{equation*}
Then $\pi$ is not fiber-wise sub-positive mixing. Let $\mu$ be the unique Gibbs measure for $f=0$ on $X$. Hence $\G=\H$.
 It is not difficult to show that $\H$ is weakly almost additive by using that $h[2^n]=h[2^n1]=n+1$ and  
$h[12^n]=h[12^n1]=n$ for $n\geq 1$ and setting $C_n=\log(3(n+1))$ in the definition of the weakly almost additive sequence.  Since $\H$ is weakly almost additive, by Theorem \ref{phd1}, $\pi\mu$ is a unique weak Gibbs for
a continuous function and it is given by $\hat h(=g \text{ in } (\ref{defineg}))$. 
 
 \end{ex}

{{Acknowledgements.}}
{The author would like to thank Professor De-Jun Feng for useful discussion.}




\end{document}